\documentclass[psamsfonts]{amsart}

\usepackage{amssymb,amsfonts,amscd}

\usepackage[colorlinks,linktocpage]{hyperref}
\hypersetup{linkcolor=[rgb]{0,0,0.715}}
\hypersetup{citecolor=[rgb]{0,0.715,0}}

\newtheorem{thm}{Theorem}[section]

\newtheorem{lem}[thm]{Lemma}

\newtheorem*{thm1}{Theorem A}

\theoremstyle{definition}

\theoremstyle{remark}
\newtheorem{rem}[thm]{Remark}

\makeatletter
\let\c@equation\c@thm
\makeatother
\numberwithin{equation}{section}

\title[]{Examples of Ricci limit spaces with \\ non-integer Hausdorff dimension}
\author[]{Jiayin Pan}
\address[Jiayin Pan]{Department of Mathematics, University of California, Santa Barbara, CA, USA.}
\email{jypan10@gmail.com}
\thanks{J. Pan is partially supported by AMS Simons travel grant.}

\author[]{Guofang Wei}
\address[Guofang Wei]{Department of Mathematics, University of California,  Santa Barbara, CA, USA.}
\email{wei@math.ucsb.edu}

\thanks{G. Wei is partially supported by NSF DMS 1811558.}
\begin{document}

\begin{abstract}
	We give the first examples of collapsing Ricci limit spaces on which the Hausdorff dimension of the singular set exceeds that of the regular set; moreover, the Hausdorff dimension of these spaces can be non-integers. This answers a question of Cheeger-Colding \cite[Page 15]{ChCo00a} about collapsing Ricci limit spaces. 	
\end{abstract}

\maketitle

Recall that for any sequence of complete $n$-dimensional Riemannian manifolds $(M_i,p_i)$ with Ricci curvature uniformly bounded from below, there is always a subsequence Gromov-Hausdorff converging to a metric space $(X,p)$. $X$ is called a Ricci limit space. The regularity theory of Ricci limit spaces has been studied extensively by Cheeger, Colding, and Naber. A point $x\in X$ is called $k$-regular if every tangent cone at $x$ is isometric to $\mathbb{R}^k$, where $k$ is an integer; otherwise, it is called singular. By the work of Cheeger-Colding \cite{ChCo97,ChCo00b} and Colding-Naber \cite{CoNa12}, $X$ has a unique integer $k\in[0,n]$ such that $\mathcal{R}^k$ has full measure and $X$ is $k$-rectifiable with respect to any limit measure of $X$, where $\mathcal{R}^k$ is the set of $k$-regular points of $X$. This integer $k$ is called the rectifiable dimension of $X$. 

When the sequence is non-collapsing $\mathrm{vol}(B_1(p_i))\ge v>0$, the limit space $X$ has rectifiable dimension $n$, which is also the Hausdorff dimension of $X$; moreover, the singular set has Hausdorff dimension at most $n-2$ \cite{ChCo97}. When the sequence is collapsing $\mathrm{vol}(B_1(p_i))\to 0$, as pointed out in  \cite[Page 15]{ChCo00a}, an important question is whether the Hausdorff dimension of the singular set could exceed that of the regular set. Equivalently, one can ask if the rectifiable dimension of $X$ equals the Hausdorff dimension of $X$ \cite[Open Problem 3.2]{Na20}. In this paper, we give answers to these open questions by providing examples of Ricci limit spaces on which the Hausdorff dimension of the singular set is larger than that of the regular set; therefore, the Hausdorff dimension of these examples is larger than their rectifiable dimension. Moreover, their Hausdorff dimension can be non-integers.

\begin{thm1} \label{thm-main}
	Given any $\beta>0$, for $n$ sufficiently large (depending on $\beta$), there is an open $n$-manifold $N$ with $\mathrm{Ric}>0$, whose asymptotic cone $(Y,y)$ satisfies the following:\\
	(1) $Y=\mathcal{S}\cup \mathcal{R}^2$, where $\mathcal{S}$ is the singular set and $\mathcal{R}^2$ is the set of $2$-regular points;\\
	(2) $\mathcal{S}$ has Hausdorff dimension $1+\beta$.\\
	When $\beta>1$, $Y$ has Hausdorff dimension $1+\beta$, larger than its rectifiable dimension.
\end{thm1}

By \cite{ChCo00b}, the above $\mathcal{R}^2$ has Hausdorff dimension $2$ and $Y$ has rectifiable dimension $2$. Thus for $\beta>1$, the Hausdorff dimension of $\mathcal{S}$ is larger than that of the regular set; also, $Y$ has Hausdorff dimension $1+\beta$, which is not an integer when $\beta \not\in\mathbb{N}$. Note that with respect to any limit measure, the regular set always has full measure \cite{ChCo97}. 
Our examples show that for a collapsing Ricci limit space, its limit measure and any dimensional Hausdorff measure may not be mutually absolutely continuous.

We point out that $Y$ is non-polar at $y$; see Remark \ref{rem_polar}. The non-polarity of $(Y,y)$ is consistent with \cite[Theorem 1.38]{ChCo00a}: if any iterated tangent cone of $Y$ is polar, then the Hausdorff dimension of $Y$ is an integer. The first examples of non-polar asymptotic cones are constructed by Menguy \cite{Men00}.

We can also use the examples in Theorem A to construct compact Ricci limit spaces with non-integer Hausdorff dimension; see Remark \ref{rem_cpt}.

Our examples are indeed closely related to some known examples \cite{Nab80,Ber86,Wei88}. Nabonnand showed that for suitable functions $f(r)$ and $h(r)$, the doubly warped product $[0,\infty)\times_f S^2 \times_h S^1$ has positive Ricci curvature \cite{Nab80}. Our construction starts with a doubly warped product $M=[0,\infty)\times_f S^{p-1}\times_h S^1$, where warping functions  $f(r), h(r)$ are the ones in \cite{Wei88} and $h(r)$ has decay rate $r^{-\beta}$;
then we take the Riemannian universal cover $N$ of $M$ as our example in Theorem A.

One key feature of the universal cover $N$ with $\pi_1(M,p)$-action is that, $d(\gamma^l\tilde{p},\tilde{p})$ is comparable to $l^{\frac{1}{1+\beta}}$, where $p\in M$ at $r=0$, $\tilde{p}\in N$ is a lift of $p$, and $\gamma$ is a generator of $\pi_1(M,p)\cong \mathbb{Z}$. With this, after passing to the asymptotic cone $(Y,y)$, the limit orbit at $y$ has Hausdorff dimension $1+\beta$ (see Section 1.3).

We would like to thank Jeff Cheeger and Aaron Naber for very helpful comments. 

\section{The examples and their geometry}

\subsection{Riemannian metric on $M$}

On $[0,\infty)\times S^{p-1}\times S^1$, where $p\ge 2$, we define a doubly warped product metric as
$$g=dr^2+f(r)^2ds_{p-1}^2+h(r)^2 ds_1^2,$$
where $ds_k^2$ is the standard metric on the unit sphere $S^{k}$. At $r=0$, we require that
$$f(0)=0,\quad f'(0)=1,\quad f''(0)=0,\quad h(0)>0,\quad h'(0)=0.$$
Then $(M,g)$ is a Riemannian manifold diffeomorpic to $\mathbb{R}^p\times S^1$.

Let $H=\partial/\partial r$, $U$ a unit vector tangent to $S^{p-1}$, and $V$ a unit vector tangent to $S^1$. The metric has Ricci curvature
\begin{align*}
	&\mathrm{Ric}(H,H)=-\dfrac{h''}{h}-(p-1)\dfrac{f''}{f},\\
	&\mathrm{Ric}(U,U)=-\dfrac{f''}{f}+\dfrac{p-2}{f^2} \left[1-(f')^2\right]-\dfrac{f'h'}{fh},\\
	&\mathrm{Ric}(V,V)=-\dfrac{h''}{h}-(p-1)\dfrac{f'h'}{fh}.
\end{align*}

We use 
$$f(r)=r(1+r^2)^{-1/4},\quad h(r)=(1+r^2)^{-\alpha},$$
where $\alpha>0$ \cite{Wei88}. On $(0,\infty)$, $f$ and $h$ satisfy
$$f'>0,\quad 1-(f')^2>0,\quad f''<0,\quad h'<0.$$
Then $\mathrm{Ric}(U,U)>0$ always holds. Also,
\begin{align*}
	\mathrm{Ric}(H,H)&> \dfrac{r^2}{(1+r^2)^2}\left[\dfrac{p-1}{4}-(2\alpha+4\alpha^2)\right]\\
	\mathrm{Ric}(V,V)&>\dfrac{\alpha r^2}{(1+r^2)^2}\left[p-(3+4\alpha)\right]
\end{align*}
When integer $p\ge \max\{4\alpha+3, 16\alpha^2+8\alpha+1\}$, $(M,g)$ has positive Ricci curvature. From now on, we denote this open Riemannian manifold as $(M,g_\alpha)$.

\subsection{Distance estimate on Riemannian universal cover $N$}

Let $(M,g_\alpha)$ as constructed above and let $(N,\tilde{g_\alpha})$ be the Riemannian universal cover of $(M,g_\alpha)$. Let $S^1(r)$ be a copy of $S^1$ in $M$ at distance $r$; it has length $L(r)=2\pi(1+r^2)^{-\alpha}.$ As the length of $S^1(r)$ decreases a lot, one expects that the length of minimal geodesics representing $\gamma^l$ grows much smaller than $l$, for $\gamma$ representing $S^1$  at $r=0$. Namely,

\begin{lem}\label{dist_base}
	Let $p\in M$ be a point with $r=0$ and let $\gamma$ be a generator of $\pi_1(M,p)\cong\mathbb{Z}$. Then for all positive integers $l\ge 9^{1+\frac{1}{2\alpha}}$,
	$$C\cdot l^{\frac{1}{1+2\alpha}}-2\le d(\gamma^l \tilde{p},\tilde{p}) \le 9\cdot l^{\frac{1}{1+2\alpha}},$$
	where $C =2 \cdot 9^{-\frac{1}{2\alpha}}$ and $\tilde{p}\in N$ is a lift of $p$.
\end{lem}

\begin{proof}
	Let $c_l$ be the minimal geodesic loop based at $p$ that represents $\gamma^l\in \pi_1(M,p)$. We estimate the length of $c_l$.
	
	Upper bound: We consider a loop $\sigma_l$ at $p$ constructed as follows. First it moves along a ray from $p$ to distance $r_l$, where $r_l$ to be determined later; then it goes around $S^1(r_l)$ circle $l$ times; then it moves back to $p$ along the ray. Then
	$$d(\gamma^l \tilde{p},\tilde{p})=\mathrm{length}(c_l)\le\mathrm{length}(\sigma_l)\le 2r_l+l\cdot L(r_l).$$
	We choose $r_l=l^{\frac{1}{1+2\alpha}}$. Then 
	$$d(\gamma^l \tilde{p},\tilde{p})\le (2+2\pi)\cdot l^{\frac{1}{1+2\alpha}}.$$
	
	Lower bound: Let $R_l>0$ be the size of $c_l$, that is, the smallest number so that $c_l$ is contained in the closed ball $\overline{B_{R_l}(p)}$. Since $h(r)$ is decreasing in $r$, we have 
	$$l\cdot L(R_l) \le \mathrm{length}(c_l)\le 9\cdot l^{\frac{1}{1+2\alpha}}.$$
	This yields
	$$R_l \ge 9^{-\frac{1}{2\alpha}}l^{\frac{1}{1+2\alpha}}-1.$$
	The desired lower bound follows from $d(\gamma^l\tilde{p},\tilde{p})\ge 2R_l$.
\end{proof}

On the other hand, we show that for $s$ large and suitable $l$, length of $c_l$ at $r=s$ is still comparable to $l\cdot \mathrm{length}S^1(s)$.

\begin{lem}\label{dist_far}
	Given any $\epsilon\in(0,1)$, $s>0$, let $l$ be any integer between $\frac{1}{2}\epsilon s(1+s^2)^{\alpha}$ and $\epsilon s (1+s^2)^{\alpha}$ (which always exists when $s$ is large). We denote $c_l$ the minimal geodesic loop based at $q$ representing $\gamma^l\in \pi_1(M,q)$, where $q\in M$ satisfies $r=s$ and $\gamma$ is a generator of $\pi_1(M,q)\cong \mathbb{Z}$. Then\\
	(1) $\epsilon\cdot C s \le \mathrm{length}(c_l)\le \epsilon \cdot 2\pi s$, where $C>0$ only depends on $\alpha$.\\
	(2) $\lim_{\epsilon\to 0} \left(\limsup_{s\to\infty} \dfrac{\mathrm{size}(c_l)}{\mathrm{length}(c_l)}\right)=0.$
\end{lem}

\begin{proof}
	(1) It is clear that
	$$\mathrm{length}(c_l)\le l\cdot 2\pi(1+s^2)^{-\alpha}\le \epsilon\cdot 2\pi s.$$
	Let $d_l$ be the size of $c_l$, that is, the smallest number so that $c_l$ is contained in $\overline{B_{d_l}}(q)$.
	Clearly, $d_l$ has a rough upper bound 
	$$d_l\le \frac{1}{2}\mathrm{length}(c_l)\le \epsilon\cdot \pi s.$$
	Then 
	\begin{align*}
		\mathrm{length}(c_l)&\ge l\cdot 2\pi(1+(s+d_l)^2)^{-\alpha}\\
		&\ge \epsilon \pi s \left( \dfrac{1+s^2}{1+(s+\epsilon \pi s)^2}\right)^\alpha\\
		&\ge \epsilon \pi s\cdot  \left(\dfrac{1}{1+\epsilon\pi}\right)^{2\alpha}\ge \dfrac{\pi}{(1+\pi)^{2\alpha}}\epsilon s.
	\end{align*}
	
	(2) To find a better estimate for $d_l$, we consider a bounded cylinder $$W=[0,s+d_l]\times S^1(s+d_l)$$ with product metric. The map
	$$F: \overline{B_{s+d_l}}(p) \to W,\quad (r,u,v)\mapsto (r,v)$$
	is distance non-increasing, where $(r,u,v)$ corresponds to the coordinate $[0,\infty)\times_f S^{p-1} \times_h S^1$ and $\overline{B_{s+d_l}}(p)$ is endowed with intrinsic metric. It is clear that $F(c_l)$ is contained in $[s,s+d_l]\times S^1 (s+d_l)$. Note that the loop $F(c_l)$ has winding number $l$ and touches both boundaries $\{s\}\times S^1(s+d_l)$ and $\{s+d_l\}\times S^1(s+d_l)$. Thus
	$$\mathrm{length}(c_l)\ge \mathrm{length}(F(c_l))\ge 2\left[d_l^2+l^2\pi^2 (1+(s+d_l)^2)^{-2\alpha}\right]^{1/2}.$$
	It follows from the upper bound of $\mathrm{length}(c_l)$ that
	$$4d_l^2+4\pi^2l^2(1+(s+d_l)^2)^{-2\alpha}\le 4\pi^2l^2(1+s^2)^{-2\alpha},$$
	$$\dfrac{d_l^2}{s^2}\le \epsilon^2 \pi^2 \left[1-\left( \dfrac{1+s^2}{1+(s+d_l)^2} \right)^{2\alpha} \right].$$
	For any sequence $s_i\to\infty$ and a sequence of integers $l_i$ between $\frac{1}{2}\epsilon s_i(1+s_i^2)^{\alpha}$ and $\epsilon s_i (1+s_i^2)^{\alpha}$ such that $\lim\limits_{i\to\infty} d_{l_i}/s_i=\delta\in [0,\epsilon\pi]$, then
	$$\delta^2 \le \epsilon^2 \pi^2 \left[1-\dfrac{1}{(1+\delta)^{4\alpha}}\right]. $$
	Thus as $\epsilon\to 0$, we see that $\delta/\epsilon\to 0$. Together with (1), we conclude that
	$$\lim\limits_{i\to\infty} \dfrac{\mathrm{size}(c_{l_i})}{\mathrm{length}(c_{l_i})}\le \lim\limits_{i\to\infty} \dfrac{d_{l_i}}{C\epsilon s_i}=\dfrac{\delta}{C\epsilon}\to 0, \text{ as }\epsilon\to 0. $$
	This proves (2).
\end{proof}

\subsection{Structure of asymptotic cones of $N$}

To better understand the structure of asymptotic cones of $N$, we will take the isometric $\Gamma$-action on $N$ into account, where $\Gamma=\pi_1(M,p)$. Let $r_i\to\infty$, passing to a subsequence if necessary, we have the following equivariant Gromov-Hausdorff convergence.
\begin{center}
	$\begin{CD}
		(r^{-1}_iN,\tilde{p},\Gamma) @>GH>> 
		(Y,y,G)\\
		@VV\pi V @VV\pi V\\
		(r^{-1}_iM,p) @>GH>> (X,x).
	\end{CD}$
\end{center}

The first author observed that the orbit $Gy$ is not geodesic in $Y$ \cite{Pan21}. Here, we give more detailed description of $(Y,y,G)$ and prove Theorem A.

From the warping functions $f$ and $h$, we see that $(X,x)$ is a half-line $[0,\infty)$ with $x=0$. Since $Y/G$ is isometric to $Z=[0,\infty)$ \cite{FY92}, we can view $Y$ as attaching orbit $Gz$ to each $z\in [0,\infty)$. Topological, $Y$ is homeomorphic to $\mathbb{R}\times [0,\infty)$, where $y=(0,0)$ and $\mathbb{R}\times \{z\}$ corresponds to the orbit $Gz$. Let $\gamma_\infty \in G$ such that $d(\gamma_\infty y,y)=1$. We label $\gamma_\infty y$ as $(1,0)\in \mathbb{R}\times [0,\infty)=Y$. Let $l_i\to \infty$ be a sequence of integers such that
$$(r_i^{-1}N,\tilde{p},\gamma^{l_i})\overset{GH}\longrightarrow (Y,y,\gamma_\infty).$$
Then for any rational number $b=m/k$, 
$$(r_i^{-1}N,\tilde{p},\gamma^{\lfloor bl_i \rfloor})\overset{GH}\longrightarrow (Y,y,g)$$
with $g$ satisfying $g^k=\gamma_\infty^m$. We label the point $gy$ as $(b,0)$. By Lemma \ref{dist_base}, there are constants $C_1,C_2>0$, only depending on $\alpha$ such that
$$C_1|b|^{\frac{1}{1+2\alpha}}\le \lim\limits_{i\to\infty}r_i^{-1}d(\gamma^{\lfloor bl_i \rfloor}\tilde{p},\tilde{p})\le C_2|b|^{\frac{1}{1+2\alpha}}.$$
Therefore, we can define a coordinate $(b,0)$, where $b\in\mathbb{R}$, for any point in $Gy$. Moreover, the coordinate satisfies
$$d((b_1-b_2,0),(0,0))=d((b_1,0),(b_2,0))$$
for all $b_1,b_2\in\mathbb{R}$. This proves the estimate below.

\begin{lem}\label{dist_Gy}
	The distance on $Gy=\mathbb{R}\times \{0\}$ satisfies
	$$C_1\cdot |b_1-b_2|^{\frac{1}{1+2\alpha}}\le d((b_1,0),(b_2,0))\le C_2\cdot |b_1-b_2|^{\frac{1}{1+2\alpha}}$$
	for all $b_1,b_2\in\mathbb{R}$, where $C_1,C_2>0$ only depend on $\alpha$.
\end{lem} 

\begin{proof}[Proof of Theorem A]
	Let $\alpha=\beta/2$. Let $(M,g_\alpha)$ be the open manifold with $\mathrm{Ric}>0$ constructed in Section 1.1 and let $(N,\tilde{g_\alpha})$ be its Riemannian universal cover. 
	
	(1) First note that $y$ is singular. In fact, let $(Y',y')$ be a tangent cone of $Y$ at $y$. By a standard diagonal argument, $(Y',y')$ is an asymptotic cone of $N$; in other words, we can find $s_i\to \infty$ such that
	$$(s_i^{-1}N,\tilde{p},\Gamma)\overset{GH}\longrightarrow (Y',y',G').$$
	The orbit $G'y'$ satisfies the distance estimate in Lemma \ref{dist_Gy}. Therefore, $(Y',y')$ is not isometric to any $\mathbb{R}^k$.
	
	Next we apply Lemma \ref{dist_far} to show that any point outside $Gy$ is $2$-regular. Let $z\in X$ and let $\tilde{z}\in Y$ be a lift of $z$ such that $d(\tilde{z},y)=d(z,x)$. Since any point outside $Gy$ is some $g\tilde{z}$, it suffices to show that $\tilde{z}$ is $2$-regular.
	
	We choose a sequence of points $\tilde{q_i}$ such that $d(\tilde{p},\tilde{q_i})=r_it$, where $t=d(z,x)$, and
	$$(r_i^{-1}N,\tilde{p},\tilde{q_i})\overset{GH}\longrightarrow (Y,y,\tilde{z}).$$
	We write $s_i=r_it$. Let $\epsilon\in(0,1)$ and let $l_i\to\infty$ be a sequence of integers between $\frac{1}{2}\epsilon s_i(1+s_i^2)^{\alpha}$ and $\epsilon s_i (1+s_i^2)^{\alpha}$. By Lemma \ref{dist_far}(1),
	$$(r_i^{-1}N,\tilde{q_i},\gamma^{l_i})\overset{GH}\longrightarrow(Y,\tilde{z},g)$$
	with $C\epsilon t \le d(g\tilde{z},\tilde{z})\le \epsilon t$, where $C\in(0,1)$ only depends on $\alpha$. Let $\sigma_i$ be a minimal geodesic from $\tilde{q_i}$ to $\gamma^{l_i}\tilde{q_i}$. After blowing down, $\sigma_i$ converges to a minimal geodesic $\sigma$ from $\tilde{z}$ to $g\tilde{z}$. Let $m$ be the midpoint of $\sigma$ and let $h\tilde{z}\in G\tilde{z}$ be a closest point in the orbit to $m$. We write $\delta=d(m,h\tilde{z})$. Then $\tau:=h^{-1}\sigma$ is a segment whose midpoint has distance $\delta$ to $\tilde{z}$.
	
	Now we change $\epsilon$. To emphasize the dependence on $\epsilon$, we will write $\tau_\epsilon$ and $\delta_\epsilon$ instead of $\tau$ and $\delta$, respectively. By our construction and Lemma \ref{dist_far}(2), minimal geodesics $\tau_\epsilon$ satisfies \\
	(i) $\mathrm{length}(\tau_\epsilon)\in [C\epsilon t,\epsilon t]$;\\
	(ii) midpoint of $\tau_\epsilon$ has distance $\delta_\epsilon$ to $\tilde{z}$, where $\delta_\epsilon/\epsilon\to 0$ as $\epsilon\to 0$.\\
	Together with Cheeger-Colding splitting theorem \cite{ChCo96}, we see that any tangent cone at $\tilde{z}$ is isometric to $\mathbb{R}^2$.
	
	(2) This follows directly from the distance estimate in Lemma \ref{dist_Gy}. In fact, the maps
	$$\phi: (Gy,d)\to (\mathbb{R},d_0),\quad (b,0)\mapsto b,$$
	$$\psi: (\mathbb{R},d_0)\to (Gy,d),\quad b\mapsto (b,0),$$
	are H\"older continuous with exponent $1+2\alpha$ and $1/(1+2\alpha)$ respectively, where $d_0$ is the Euclidean distance on $\mathbb{R}$. Thus
	$$\dim_H (\mathbb{R},d_0)\le \dfrac{1}{1+2\alpha}\dim_H(Gy,d),$$ $$\dim_H(Gy,d)\le (1+2\alpha)\dim_H(\mathbb{R},d_0).$$
	We conclude that $Gy$ has Hausdorff dimension $1+2\alpha$.
\end{proof}

We end our paper with some remarks related to our examples.

\begin{rem}
	One can also derive that every point $z\not\in Gy$ is regular from the regularity theory. If $z$ is singular, then any point in $Gz$ is also singular; this would contradict the path-connectedness of the regular set \cite{CoNa12}.
\end{rem}

\begin{rem}\label{rem_polar}
	For any asymptotic cone $(Y,y)$ of $(N,\tilde{p})$, $Y$ is non-polar at $y$. In fact, we consider a minimal geodesic $\sigma$ from $y$ to a point $gy\not= y$. We claim that $\sigma$ cannot be extended further across $gy$ as a minimal geodesic. Suppose that we can extend $\sigma$ to $\overline{\sigma}$. First note that the extended part is not contained in $Gy$; otherwise, it would have infinite length by Lemma \ref{dist_Gy}. Thus the minimal geodesic $\overline{\sigma}$ contains two regular points with a singular point $gy$ in between; this would contradict the H\"older continuity of tangent cones \cite{CoNa12,Deng20}.
\end{rem}

\begin{rem}
	If the warping function $h$ has logarithm decay $\ln^{-\alpha}(r)$ or it decays to a positive number $c>0$, then $(Y,y)$, the corresponding asymptotic cone of the universal cover, splits isometrically as $\mathbb{R}\times [0,\infty)$; see \cite{Pan21}.
\end{rem}

\begin{rem}
	We suspect that the structure of fundamental group $\pi_1(M)$ is related to the structure of singular set of $Y$: if $\pi_1(M)$ contains a torsion-free nilpotent subgroup of class $\ge 2$, is it true that for some asymptotic cone $(Y,y,G)$ of the universal cover $(N,\tilde{p},\pi_1(M,p))$, its orbit $Gy$ must be non-rectifiable? The first author proved that for such a fundamental group, in some $(Y,y,G)$ there is a geodesic connecting two points in $Gy$ but definite away from $Gy$; see \cite{Pan21,Pan20}.
\end{rem}

\begin{rem}\label{rem_cpt}We can construct a compact Ricci limit space with same feature as in Theorem A as follows. Let $\beta=2\alpha$, $(M,g_\alpha)$, and $(N,\tilde{g_\alpha})$ as constructed before. For a sequence $r_i\to\infty$, we choose $\gamma_i=\gamma^{l_i}\in \Gamma$, where $l_i\to\infty$, so that 
	$$(r_i^{-1}N,\tilde{p},\gamma^{l_i},\Gamma)\overset{GH}\longrightarrow (Y,y,\gamma_\infty,G),$$
	where $\gamma_\infty \in G$ with $d(\gamma_\infty y,y)=1$. Let  $(\overline{M_i},\bar{p_i})= (N,\tilde{p})/ \langle\gamma_i \rangle$. Then
	$$(r_i^{-1}\overline{M_i},\bar{p_i})\overset{GH}\longrightarrow (Y/\mathbb{Z},\bar{y}),$$
	where $\mathbb{Z}=\langle \gamma_\infty \rangle$ is a discrete subgroup of $G$ and the limit space $Y/\mathbb{Z}$ has Hausdorff dimension $1+\beta$ when $\beta>1$.
	
	Next, we glue two copies of a bounded region in $r_i^{-1}\overline{M_i}$ along the boundary to obtain a compact Ricci limit space. Let 
	$K_i=\pi_i^{-1}(B_{r_i}(p))\subseteq \overline{M_i}$, where $\pi_i:\overline{M_i}\to M$ is the covering map.
	The principal eigenvalues of the second fundamental from of $\partial B_r(p)$ are $\tfrac{f'(r)}{f(r)}, \tfrac{h'(r)}{h(r)}$, which decay like $r^{-1}$ as $r \to \infty$. Thus the second fundamental forms of $r_i^{-1}(\partial K_i)$ are  uniformly bounded. Now one can extending the boundary of $K_i$ with definite length to have totally geodesic boundary, then we can double it to have closed manifolds with Ricci curvature uniformly bounded from below and diameter uniformly bounded from above. The Hausdorff dimension of its Gromov-Hausdorff limit is also $1+ \beta$ when $\beta>1$.
\end{rem}


\begin{thebibliography}{10}
	
	\bibitem{Ber86}
	Lionel B\'{e}rard-Bergery.
	\newblock Quelques exemples de vari\'{e}t\'{e}s riemanniennes compl\`etes non
	compactes \`a courbure de {R}icci positive.
	\newblock {\em C. R. Acad. Sci. Paris S\'{e}r. I Math.}, 302(4):159--161, 1986.
	
	\bibitem{ChCo96}
	Jeff Cheeger and Tobias~H. Colding.
	\newblock Lower bounds on {R}icci curvature and the almost rigidity of warped
	products.
	\newblock {\em Ann. of Math. (2)}, 144(1):189--237, 1996.
	
	\bibitem{ChCo97}
	Jeff Cheeger and Tobias~H. Colding.
	\newblock On the structure of spaces with {R}icci curvature bounded below. i.
	\newblock {\em J. Differential Geom.}, 46(3):406--480, 1997.
	
	\bibitem{ChCo00a}
	Jeff Cheeger and Tobias~H. Colding.
	\newblock On the structure of spaces with {R}icci curvature bounded below. ii.
	\newblock {\em J. Differential Geom.}, 54(1):13--35, 2000.
	
	\bibitem{ChCo00b}
	Jeff Cheeger and Tobias~H. Colding.
	\newblock On the structure of spaces with {R}icci curvature bounded below. iii.
	\newblock {\em J. Differential Geom.}, 54(1):37--74, 2000.
	
	\bibitem{CoNa12}
	Tobias~H. Colding and Aaron Naber.
	\newblock {Sharp Hölder continuity of tangent cones for spaces with a lower
		Ricci curvature bound and applications}.
	\newblock {\em Annals of Mathematics}, 176(2):1173--1229, 2012.
	
	\bibitem{Deng20}
	Qin {Deng}.
	\newblock {H\"older continuity of tangent cones in RCD(K,N) spaces and
		applications to non-branching}.
	\newblock {\em arXiv:2009.07956}, 2020.
	
	\bibitem{FY92}
	Kenji Fukaya and Takao Yamaguchi.
	\newblock The fundamental groups of almost nonnegatively curved manifolds.
	\newblock {\em Annals of Mathematics}, 136(2):253--333, 1992.
	
	\bibitem{Men00}
	Xavier Menguy.
	\newblock Examples of nonpolar limit spaces.
	\newblock {\em Amer. J. Math.}, 122(5):927--937, 2000.
	
	\bibitem{Na20}
	Aaron Naber.
	\newblock Conjectures and open questions on the structure and regularity of
	spaces with lower {R}icci curvature bounds.
	\newblock {\em SIGMA Symmetry Integrability Geom. Methods Appl.}, 16:Paper No.
	104, 8, 2020.
	
	\bibitem{Nab80}
	Philippe Nabonnand.
	\newblock Sur les vari\'{e}t\'{e}s riemanniennes compl\`etes \`a courbure de
	{R}icci positive.
	\newblock {\em C. R. Acad. Sci. Paris S\'{e}r. A-B}, 291(10):A591--A593, 1980.
	
	\bibitem{Pan20}
	Jiayin {Pan}.
	\newblock {Nonnegative {R}icci curvature and escape rate gap}.
	\newblock {\em arXiv:2009.00226}, 2020.
	
	\bibitem{Pan21}
	Jiayin Pan.
	\newblock On the escape rate of geodesic loops in an open manifold with
	nonnegative {R}icci curvature.
	\newblock {\em Geom. \& Topol.}, 25(2):1059--1085, 2021.
	
	\bibitem{Wei88}
	Guofang Wei.
	\newblock Examples of complete manifolds of positive {R}icci curvature with
	nilpotent isometry groups.
	\newblock {\em Bull. Amer. Math. Soc. (N.S.)}, 19(1):311--313, 1988.
	
\end{thebibliography}
\end{document}